\documentclass{amsart}%
\usepackage{amsmath}
\usepackage{amssymb}
\usepackage{amsfonts}
\usepackage{graphicx}%
\newtheorem{theorem}{Theorem}[section]
\theoremstyle{plain}

\newtheorem{corollary}[theorem]{Corollary}

\newtheorem{lemma}[theorem]{Lemma}

\newtheorem{proposition}[theorem]{Proposition}

\numberwithin{equation}{section}
\begin{document}
\title[Semiclassical solutions to a Schr\"{o}dinger-Newton system]{Intertwining semiclassical solutions to a Schr\"{o}dinger-Newton system}
\author{Silvia Cingolani}
\address{Dipartimento di Matematica, Politecnico di Bari, via Orabona 4, 70125 Bari, Italy.}
\email{s.cingolani@poliba.it}
\author{M\'{o}nica Clapp}
\address{Instituto de Matem\'{a}ticas, Universidad Nacional Aut\'{o}noma de M\'{e}xico,
Circuito Exterior, C.U., 04510 M\'{e}xico D.F., Mexico.}
\email{mclapp@matem.unam.mx}
\author{Simone Secchi}
\address{Dipartimento di Matematica ed Applicazioni, Universit\`{a} di Milano-Bicocca,
via Cozzi 53, 20125 Milano, Italy.}
\email{Simone.Secchi@unimib.it}
\thanks{S. Cingolani is supported by the MIUR project \textit{Variational and
topological methods in the study of nonlinear phenomena} (PRIN 2007).}
\thanks{M. Clapp is supported by CONACYT grant 129847 and PAPIIT grant IN101209 (Mexico).}
\maketitle

\begin{abstract}
We study the problem
\[%
\begin{cases}
\left(  -\varepsilon\mathrm{i}\nabla+A(x)\right)  ^{2}u+V(x)u=\varepsilon
^{-2}\left(  \frac{1}{|x|}\ast|u|^{2}\right)  u,\\
u\in L^{2}(\mathbb{R}^{3},\mathbb{C}),\text{ \ \ \ \ }\varepsilon\nabla
u+\mathrm{i}Au\in L^{2}(\mathbb{R}^{3},\mathbb{C}^{3}),
\end{cases}
\]
where $A\colon\mathbb{R}^{3}\rightarrow\mathbb{R}^{3}$ is an exterior magnetic
potential, $V\colon\mathbb{R}^{3}\rightarrow\mathbb{R}$ is an exterior
electric potential, and $\varepsilon$ is a small positive number. If $A=0$ and
$\varepsilon=\hbar$ is Planck's constant this problem is equivalent to the
Schr\"{o}dinger-Newton equations proposed by Penrose in \cite{pe2}\ to
describe his view that quantum state reduction occurs due to some
gravitational effect. We assume that $A$ and $V$ are compatible with the
action of a group $G$ of linear isometries of $\mathbb{R}^{3}$. Then, for any
given homomorphism $\tau:G\rightarrow\mathbb{S}^{1}$ into the unit complex
numbers, we show that there is a combined effect of the symmetries and the
potential $V$ on the number of semiclassical solutions $u:\mathbb{R}%
^{3}\rightarrow\mathbb{C}$ which satisfy $u(gx)=\tau(g)u(x)$ for all $g\in G$,
$x\in\mathbb{R}^{3}$. We also study the concentration behavior of these
solutions as $\varepsilon\rightarrow0.\medskip$



\end{abstract}

\section{Introduction}

The \emph{Schr\"{o}dinger-Newton equations} were proposed by Penrose
\cite{pe2} to describe his view that quantum state reduction is a phenomenon
that occurs because of some gravitational influence. They consist of a system
of equations obtained by coupling together the linear Schr\"{o}dinger equation
of quantum mechanics with the Poisson equation from Newtonian mechanics. For a
single particle of mass $m$ this system has the form
\begin{equation}%
\begin{cases}
-\frac{\hbar^{2}}{2m}\Delta\psi+V(x)\psi+U\psi=0,\\
-\Delta U+4\pi\kappa|\psi|^{2}=0,
\end{cases}
\label{sys:s-n}%
\end{equation}
where $\psi$ is the complex wave function, $U$ is the gravitational potential
energy, $V$ is a given potential, $\hbar$ is Planck's constant, and
$\kappa:=\mathrm{G}m^{2}$, $\mathrm{G}$ being Newton's constant. According to
Penrose, the solutions $\psi$ of this system are the \emph{basic stationary
states }into which a superposition of such states is to decay within a certain
timescale, cf. \cite{pe1, pe2, mpt, tm, pe3}.

After rescaling by%
\[
\psi(x)=\frac{1}{\hbar}\frac{\hat{\psi}(x)}{\sqrt{2\kappa m}},\quad
V(x)=\frac{1}{2m}\hat{V}(x),\quad U(x)=\frac{1}{2m}\hat{U}(x),
\]
system~\eqref{sys:s-n} can be written as
\begin{equation}%
\begin{cases}
-\hbar^{2}\Delta\hat{\psi}+\hat{V}(x)\hat{\psi}+\hat{U}\hat{\psi}=0,\\
-\hbar^{2}\Delta\hat{U}+4\pi|\hat{\psi}|^{2}=0.
\end{cases}
\label{sys:s-n-red}%
\end{equation}
The second equation in~\eqref{sys:s-n-red} can be explicitly solved with
respect to $\hat{U}$, so this system is equivalent to the single nonlocal
equation
\begin{equation}
-\hbar^{2}\Delta\hat{\psi}+\hat{V}(x)\hat{\psi}=\frac{1}{\hbar^{2}}\left(
\int_{\mathbb{R}^{3}}\frac{|\hat{\psi}(\xi)|^{2}}{|x-\xi|}d\xi\right)
\hat{\psi}\quad\text{in $\mathbb{R}^{3}$}. \label{eq:1.3}%
\end{equation}

We shall consider a more general equation having a similar structure, namely
\begin{equation}
\left(  -{\varepsilon}\mathrm{i}\nabla+A(x)\right)  ^{2}u+V(x)u=\frac
{1}{\varepsilon^{2}}\left(  \frac{1}{\left\vert x\right\vert }\ast
|u|^{2}\right)  u\quad\text{in $\mathbb{R}^{3}$}, \label{eq:Hartree}%
\end{equation}
where $A\colon\mathbb{R}^{3}\rightarrow\mathbb{R}^{3}$ is an exterior magnetic
potential, $\mathrm{i}$ is the imaginary unit and $\ast$ denotes the
convolution operator. We are interested in semiclassical states, i.e. in
solutions of this equation for $\varepsilon\rightarrow0.$

The existence of one solution can be traced back to Lions' paper~\cite{l}. In
the nonmagnetic case $A=0$ equation (\ref{eq:Hartree}) and related equations
have been investigated by many authors, see e.g. \cite{a, fl, fty, hmt, lieb,
l2, mz, mpt, n,secchi,t, tm} and the references therein. Recently, Wei and
Winter \cite{wei} showed the existence of positive multibump solutions which
concentrate at local minima, local maxima or nondegenerate critical points of
the potential $V$ as $\varepsilon\rightarrow0$. The magnetic case $A\neq0$ was
recently studied in \cite{css} where it was shown that equation
(\ref{eq:Hartree}) has a family of solutions having multiple concentration
regions located around the (possibly degenerate) minima of $V$.

In this paper we consider the situation where $A$ and $V$ are symmetric and we
look for semiclassical solutions of equation (\ref{eq:Hartree}) having
specific symmetries. The absolute value of the solutions we obtain
concentrates at points which need not be local extrema, nor nondegenerate
critical points of $V$ (in fact, we do not even assume that $V$ is
differentiable). We state our main results in the following section and give
some explicit examples.

\section{Statement of results}

\subsection{The results}

Let $G$ be a closed subgroup of the group $O(3)$ of linear isometries of
$\mathbb{R}^{3},$ $A:\mathbb{R}^{3}\rightarrow\mathbb{R}^{3}$ be a $C^{1}%
$-function, and $V:\mathbb{R}^{3}\rightarrow\mathbb{R}$ be a bounded
continuous function with $\inf_{\mathbb{R}^{3}}V>0$, which satisfy
\begin{equation}
A(gx)=gA(x)\text{ \ \ and \ \ }V(gx)=V(x)\quad\text{for all $g\in G$,
$x\in\mathbb{R}^{3}$}. \label{G}%
\end{equation}
Given a continuous homomorphism of groups $\tau:G\rightarrow\mathbb{S}^{1}$
into the group $\mathbb{S}^{1}$ of unit complex numbers, we look for solutions
to the problem%
\begin{equation}%
\begin{cases}
\left(  -\varepsilon\mathrm{i}\nabla+A\right)  ^{2}u+V(x)u=\varepsilon
^{-2}\left(  \frac{1}{\left\vert x\right\vert }\ast|u|^{2}\right)  u,\\
u\in L^{2}(\mathbb{R}^{3},\mathbb{C}),\\
\varepsilon\nabla u+\mathrm{i}Au\in L^{2}(\mathbb{R}^{3},\mathbb{C}^{3}),
\end{cases}
\label{prob}%
\end{equation}
which satisfy%
\begin{equation}
u(gx)=\tau(g)u(x)\text{ \ \ \ for all }g\in G,\text{ }x\in\mathbb{R}^{3},
\label{tau-inv}%
\end{equation}
This implies that the absolute value $\left\vert u\right\vert $ of $u$ is
$G$-invariant, i.e.%
\[
\left\vert u(gx)\right\vert =\left\vert u(x)\right\vert \quad\text{for all
$g\in G$, $x\in\mathbb{R}^{3}$},
\]
whereas the phase of $u(gx)$ is that of $u(x)$ multiplied by $\tau(g).$ A
concrete example is given in subsection \ref{example} below.

Note that if $u$ satisfies (\ref{prob}) and (\ref{tau-inv}) then
$e^{\mathrm{i}\theta}u$ satisfies (\ref{prob}) and (\ref{tau-inv}) for every
$\theta\in\mathbb{R}.$ We shall say that $u$ and $v$ are \emph{geometrically
distinct} if $e^{\mathrm{i}\theta}u\neq v$ for all $\theta\in\mathbb{R}.$

We introduce some notation. For $x\in\mathbb{R}^{3},$ we denote by $Gx$ the
$G$-orbit of $x$ and by $G_{x}$ the $G$-isotropy subgroup of $x,$ i.e.%
\[
Gx:=\left\{  gx\colon g\in G\right\}  ,\quad G_{x}:=\{g\in G:gx=x\}.
\]
A subset $X$ of $\mathbb{R}^{3}$\ is $G$-invariant if $Gx\subset X$ for every
$x\in X.$ The $G$-orbit space of $X$ is the set
\[
X/G:=\{Gx:x\in X\}
\]
of $G$-orbits of $X$ with the quotient topology.

Let $\#Gx$ denote the cardinality of $Gx,$ and define%
\[
\ell_{G,V}:=\inf_{x\in\mathbb{R}^{3}}(\#Gx)V^{3/2}(x),
\]%
\[
M_{\tau}:=\left\{  x\in\mathbb{R}^{3}:(\#Gx)V^{3/2}(x)=\ell_{G,V},\text{
}G_{x}\subset\ker\tau\right\}  .
\]
Assumption (\ref{G}) implies that $M_{\tau}$ is $G$-invariant. Observe that
the points of $M_{\tau}$ need not be neither local minima nor local maxima of
$V$.

Given $\rho>0$ we set $B_{\rho}M_{\tau}:=\{x\in\mathbb{R}^{3}:\,$%
dist$(x,M_{\tau})\leq\rho\},$ and write
\[
\text{cat}_{B_{\rho}M_{\tau}/G}(M_{\tau}/G)
\]
for the Lusternik-Schnirelmann category of $M_{\tau}/G$ in $B_{\rho}M_{\tau
}/G.$

Finally, we denote by $E_{1}$ the least energy of a nontrivial solution to
problem%
\begin{equation}%
\begin{cases}
-\Delta u+u=(\frac{1}{\left\vert x\right\vert }\ast u^{2})u,\\
u\in H^{1}(\mathbb{R}^{3},\mathbb{R}).
\end{cases}
\label{limit}%
\end{equation}
We shall prove the following results.

\begin{theorem}
\label{mainthm1} Assume there exists $\alpha>0$ such that the set
\begin{equation}
\left\{  x\in\mathbb{R}^{3}:(\#Gx)V^{3/2}(x)\leq\ell_{G,V}+\alpha\right\}
\label{comp}%
\end{equation}
is compact. Then, given $\rho,\delta>0,$ there exists $\widehat{\varepsilon
}>0$ such that, for every $\varepsilon\in(0,\widehat{\varepsilon})$, problem
\emph{(\ref{prob})} has at least%
\[
\text{\emph{cat}}_{B_{\rho}M_{\tau}/G}(M_{\tau}/G)
\]
geometrically distinct solutions $u$ which satisfy \emph{(\ref{tau-inv})} and%
\begin{equation}
\left\vert \frac{1}{4}\int_{\mathbb{R}^{3}}\left(  \frac{1}{\left\vert
x\right\vert }\ast\left\vert u\right\vert ^{2}\right)  \left\vert u\right\vert
^{2}-\varepsilon^{5}\ell_{G,V}E_{1}\right\vert <\varepsilon^{5}\delta.
\label{energy}%
\end{equation}

\end{theorem}

The last inequality says that the energy of the solutions is arbitrarily close
to $\varepsilon^{3}\ell_{G,V}E_{1}$ for $\varepsilon$ small enough. So
considering different groups $G$ and $G^{\prime}$ for which $\ell_{G,V}%
\neq\ell_{G^{\prime},V}$ will lead to solutions with energy in disjoint ranges.

For $u\in H^{1}(\mathbb{R}^{3},\mathbb{R})$ set
\[
\left\Vert u\right\Vert _{\varepsilon}^{2}:=\int_{\mathbb{R}^{3}}%
(\varepsilon^{2}\left\vert \nabla u\right\vert ^{2}+u^{2}).
\]
The following theorem describes the module of the solutions given by Theorem
\ref{mainthm1} as $\varepsilon\rightarrow0.$

\begin{theorem}
\label{concentrationthm}Let $u_{n}$ be a solution to problem \emph{(\ref{prob}%
)} which satisfies \emph{(\ref{tau-inv})} and \emph{(\ref{energy})} for
$\varepsilon=\varepsilon_{n}>0$, $\delta=\delta_{n}>0$. Assume $\varepsilon
_{n}\rightarrow0$ and $\delta_{n}\rightarrow0.$ Then, after passing to a
subsequence, there exists a sequence $(\xi_{n})$ in $\mathbb{R}^{3}$ such that
$\xi_{n}\rightarrow\xi\in M_{\tau}$, $G_{\xi_{n}}=G_{\xi},$ and%
\[
\varepsilon_{n}^{-3}\left\Vert \left\vert u_{n}\right\vert -%
{\textstyle\sum\limits_{g\xi_{n}\in G\xi_{n}}}
\omega_{\xi}\left(  \frac{\cdot-g\xi_{n}}{\varepsilon_{n}}\right)  \right\Vert
_{\varepsilon_{n}}^{2}\rightarrow0,
\]
where $\omega_{\xi}$ is the unique ground state of problem%
\[
-\Delta u+V(\xi)u=\left(  \frac{1}{\left\vert x\right\vert }\ast u^{2}\right)
u,\hspace{0.3in}u\in H^{1}(\mathbb{R}^{3},\mathbb{R}),
\]
which is positive and radially symmetric with respect to the origin.
\end{theorem}

Next, we give an example which illustrates our results.

\subsection{Rotationally invariant potentials}

\label{example}Let $\mathbb{S}^{1}$ act on $\mathbb{R}^{3}\equiv
\mathbb{C}\times\mathbb{R}$ by $e^{\mathrm{i}\theta}(z,t):=(e^{\mathrm{i}%
\theta}z,t),$ and let $A$ and $V$ satisfy assumption (\ref{G}) for the cyclic
group $G_{m}$ generated by $e^{2\pi\mathrm{i}/m},$ for some $m\in\mathbb{N}$.
For example, the standard magnetic potential $A(x_{1},x_{2},x_{3}%
):=(-x_{2},x_{1},0)$ associated to the constant magnetic field $B(x)=(0,0,2)$
has this property for every $m$.

For each $j=0,1,\ldots,m-1$ we look for solutions to problem (\ref{prob})
which satisfy%
\begin{equation}
u(e^{2\pi\mathrm{i}/m}z,t)=e^{2\pi\mathrm{i}j/m}u(z,t)\hspace{0.2in}\text{for
all \ }(z,t)\in\mathbb{C}\times\mathbb{R}. \label{symex}%
\end{equation}
Solutions of this type arise in a natural way in some problems where the
magnetic potential is singular and the topology of the domain produces an
Aharonov-Bohm type effect, cf. \cite{at, cs}. Taking $\tau_{j}(g):=g^{j}$ we
see that these are solutions of the type furnished by Theorem \ref{mainthm1}.

If $V$ satisfies
\begin{equation}
V_{0}:=\inf_{x\in\mathbb{R}^{3}}V<\liminf_{\left\vert x\right\vert
\rightarrow\infty}V(x)\text{ \ \ \ \ and \ \ \ }mV_{0}^{3/2}<\inf
_{t\in\mathbb{R}}V^{3/2}(0,t), \label{V}%
\end{equation}
then assumption (\ref{comp}) in Theorem \ref{mainthm1} is satisfied,
$\ell_{G_{m},V}=mV_{0}^{3/2}$ and $M_{\tau}$ is simply the set of minima of
$V,$%
\[
M=\left\{  x\in\mathbb{R}^{3}:V(x)=V_{0}\right\}  .
\]
Thus, for each $j=0,1,\ldots,m-1$ and $\rho$, $\delta>0$, Theorem
\ref{mainthm1}\ yields at least \ cat$_{B_{\rho}M/G_{m}}(M/G_{m})$
\ geometrically distinct solutions to problem (\ref{prob}) satisfying
(\ref{symex}) and (\ref{energy}), for $\varepsilon$ small enough.

For each $k$ dividing $m$ the potentials $A$ and $V$ satify assumption
(\ref{G}) for $G_{k}$ and $V$ satisfies (\ref{V})$\ $with $k$ instead of $m$.
Property (\ref{energy}) implies that the solutions obtained for $G_{k}$ are
different from those for $G_{m}$ if $k\neq m$ and $\varepsilon$ is small enough.

\bigskip

This paper is organized as follows. In section \ref{secvarprob} we discuss the
variational problem related to the existence of solutions to problem
(\ref{prob}) satisfying (\ref{tau-inv}). We also outline the strategy for
proving Theorem \ref{mainthm1}. Sections \ref{embedding} and \ref{baryorbit}
are devoted to the construction of an \emph{entrance map} and a \emph{local
baryorbit map} which will help us estimate the Lusternik-Schnirelmann category
of a suitable sublevel set of the variational functional for $\varepsilon$
small enough. Finally, in section \ref{secproofs} we prove Theorems
\ref{mainthm1} and \ref{concentrationthm}.

\section{The variational problem}

\label{secvarprob}Set $\nabla_{\varepsilon,A}u:=\varepsilon\nabla
u+\mathrm{i}Au$ and consider the real Hilbert space%
\[
H_{\varepsilon,A}^{1}(\mathbb{R}^{3},\mathbb{C}):=\{u\in L^{2}(\mathbb{R}%
^{3},\mathbb{C}):\nabla_{\varepsilon,A}u\in L^{2}(\mathbb{R}^{3}%
,\mathbb{C}^{3})\}
\]
with the scalar product
\begin{equation}
\left\langle u,v\right\rangle _{\varepsilon,A,V}:=\operatorname{Re}%
\int_{\mathbb{R}^{3}}\left(  \nabla_{\varepsilon,A}u\cdot\overline
{\nabla_{\varepsilon,A}v}+V(x)u\overline{v}\right)  . \label{sp}%
\end{equation}
We write
\[
\left\Vert u\right\Vert _{\varepsilon,A,V}:=\left(  \int_{\mathbb{R}^{3}%
}\left(  \left\vert \nabla_{\varepsilon,A}u\right\vert ^{2}+V(x)\left\vert
u\right\vert ^{2}\right)  \right)  ^{1/2}%
\]
for the corresponding norm.

If $u\in H_{\varepsilon,A}^{1}(\mathbb{R}^{3},\mathbb{C}),$ then $\left\vert
u\right\vert \in H^{1}(\mathbb{R}^{3},\mathbb{R})$ and%
\begin{equation}
\varepsilon\left\vert \nabla|u(x)|\right\vert \leq\left\vert \varepsilon\nabla
u(x)+\mathrm{i}A(x)u(x)\right\vert \text{\quad for a.e. $x\in\mathbb{R}^{3}$}.
\label{di}%
\end{equation}
This is called the diamagnetic inequality \cite{ll}. Set
\[
\mathbb{D}(u):=\int_{\mathbb{R}^{3}}\int_{\mathbb{R}^{3}}\frac{|u(x)|^{2}%
|u(y)|^{2}}{\left\vert x-y\right\vert }\,dxdy.
\]
The standard Hardy--Littlewood--Sobolev inequality \cite[Theorem 4.3]{ll}
yields%
\begin{equation}
\left\vert \int_{\mathbb{R}^{3}}\!\int_{\mathbb{R}^{3}}\frac{f(x)h(y)}%
{|x-y|}\,dx\,dy\right\vert \leq C\Vert f\Vert_{L^{6/5}(\mathbb{R}^{3})}\Vert
h\Vert_{L^{6/5}(\mathbb{R}^{3})} \label{eq:HLS}%
\end{equation}
for all $f,h\in L^{6/5}(\mathbb{R}^{3}),$ where $C$ is a positive constant
independent of $f$ and $h$. In particular,%
\begin{equation}
\mathbb{D}(u)\leq C\Vert u\Vert_{L^{12/5}(\mathbb{R}^{3})}^{4} \label{D}%
\end{equation}
for every $u\in H_{\varepsilon,A}^{1}(\mathbb{R}^{3},\mathbb{C}).$

The energy functional $J_{\varepsilon,A,V}:H_{\varepsilon,A}^{1}%
(\mathbb{R}^{3},\mathbb{C})\rightarrow\mathbb{R}$ associated to problem
(\ref{prob}), defined by
\[
J_{\varepsilon,A,V}(u):=\frac{1}{2}\left\Vert u\right\Vert _{\varepsilon
,A,V}^{2}-\frac{1}{4\varepsilon^{2}}\mathbb{D}(u),
\]
is of class $C^{2}$, and its derivative is given by%
\[
J_{\varepsilon,A,V}^{\prime}(u)v:=\left\langle u,v\right\rangle _{\varepsilon
,A,V}-\frac{1}{\varepsilon^{2}}\operatorname{Re}\int_{\mathbb{R}^{3}}\left(
\frac{1}{\left\vert x\right\vert }\ast\left\vert u\right\vert ^{2}\right)
u\overline{v}.
\]
Therefore, the solutions to problem (\ref{prob}) are the critical points of
$J_{\varepsilon,A,V}.$ We write $\nabla_{\varepsilon}J_{\varepsilon,A,V}(u)$
for the gradient of $J_{\varepsilon,A,V}$ at $u$ with respect to the scalar
product (\ref{sp}).

The action of $G$ on $H_{\varepsilon,A}^{1}(\mathbb{R}^{3},\mathbb{C})$
defined by $(g,u)\mapsto u_{g},$ where%
\[
(u_{g})(x):=\tau(g)u(g^{-1}x),
\]
satisfies%
\[
\left\langle u_{g},v_{g}\right\rangle _{\varepsilon,A,V}=\left\langle
u,v\right\rangle _{\varepsilon,A,V}\text{ \ \ \ and \ \ \ }\mathbb{D}%
(u_{g})=\mathbb{D}(u)
\]
for all $g\in G,$ $u,v\in H_{\varepsilon,A}^{1}(\mathbb{R}^{3},\mathbb{C}).$
Hence, $J_{\varepsilon,A,V}$ is $G$-invariant. By the principle of symmetric
criticality \cite{p, w}, the critical points of the restriction of
$J_{\varepsilon,A,V}$ to the fixed point space of this $G$-action, denoted by
\begin{align*}
H_{\varepsilon,A}^{1}(\mathbb{R}^{3},\mathbb{C})^{\tau}  &  =\left\{  u\in
H_{\varepsilon,A}^{1}(\mathbb{R}^{3},\mathbb{C}):u_{g}=u\right\} \\
&  =\left\{  u\in H_{\varepsilon,A}^{1}(\mathbb{R}^{3},\mathbb{C}%
):u(gx)=\tau(g)u(x)\text{ \ }\forall x\in\mathbb{R}^{3},\text{ }g\in
G\right\}  ,
\end{align*}
are the solutions to problem (\ref{prob}) which satisfy (\ref{tau-inv}). Those
which are nontrivial lie on the \emph{Nehari manifold}%
\[
\mathcal{N}_{\varepsilon,A,V}^{\tau}:=\left\{  u\in H_{\varepsilon,A}%
^{1}(\mathbb{R}^{3},\mathbb{C})^{\tau}:u\neq0,\ \varepsilon^{2}\left\Vert
u\right\Vert _{\varepsilon,A,V}^{2}=\mathbb{D}(u)\right\}  ,
\]
which is a $C^{2}$-manifold radially diffeomorphic to the unit sphere in
$H_{\varepsilon,A}^{1}(\mathbb{R}^{3},\mathbb{C})^{\tau}.$ The critical points
of the restriction of $J_{\varepsilon,A,V}$ to $\mathcal{N}_{\varepsilon
,A,V}^{\tau}$ are precisely the nontrivial solutions to (\ref{prob}) which
satisfy (\ref{tau-inv}).

The radial projection $\pi_{\varepsilon,A,V}:H_{\varepsilon,A}^{1}%
(\mathbb{R}^{3},\mathbb{C})^{\tau}\setminus\{0\}\rightarrow\mathcal{N}%
_{\varepsilon,A,V}^{\tau}$ is given by%
\begin{equation}
\pi_{\varepsilon,A,V}(u):=\frac{\varepsilon\left\Vert u\right\Vert
_{\varepsilon,A,V}}{\sqrt{\mathbb{D}(u)}}u. \label{radproj}%
\end{equation}
Note that%
\begin{equation}
J_{\varepsilon,A,V}(\pi_{\varepsilon,A,V}(u))=\frac{\varepsilon^{2}\left\Vert
u\right\Vert _{\varepsilon,A,V}^{4}}{4\mathbb{D}(u)}\text{ \ \ \ \ for all
}u\in H_{\varepsilon,A}^{1}(\mathbb{R}^{3},\mathbb{C})^{\tau}\setminus\{0\}.
\label{JpiA}%
\end{equation}

Recall that $J_{\varepsilon,A,V}:\mathcal{N}_{\varepsilon,A,V}^{\tau
}\rightarrow\mathbb{R}$ is said to satisfy the \emph{Palais-Smale condition}
$(PS)_{c}$ at the level $c$ if every sequence $(u_{n})$ such that%
\[
u_{n}\in\mathcal{N}_{\varepsilon,A,V}^{\tau},\text{ \ \ \ \ }J_{\varepsilon
,A,V}(u_{n})\rightarrow c,\text{ \ \ \ \ }\nabla_{\mathcal{N}_{\varepsilon
,A,V}^{\tau}}J_{\varepsilon,A,V}(u_{n})\rightarrow0,
\]
contains a convergent subsequence. Here $\nabla_{\mathcal{N}_{\varepsilon
,A,V}^{\tau}}J_{\varepsilon,A,V}(u)$ denotes the orthogonal projection of
$\nabla_{\varepsilon}J_{\varepsilon,A,V}(u)$ onto the tangent space to
$\mathcal{N}_{\varepsilon,A,V}^{\tau}$ at $u$. The following holds.

\begin{proposition}
\label{palaissmale} For every $\varepsilon>0$, the functional $J_{\varepsilon
,A,V}:\mathcal{N}_{\varepsilon,A,V}^{\tau}\rightarrow\mathbb{R}$ satisfies
$(PS)_{c}$ at each level
\[
c<\varepsilon^{3}\min_{x\in\mathbb{R}^{3}\setminus\{0\}}(\#Gx)V_{\infty}%
^{3/2}E_{1},
\]
where $V_{\infty}:=\liminf_{\left\vert x\right\vert \rightarrow\infty}V(x).$
\end{proposition}

\begin{proof}
This was proved in \cite{ccs2} for $\varepsilon=1.$ For $\varepsilon>0$ the
assertion follows after performing the change of variable $u_{\varepsilon
}(x):=u(\varepsilon x)$ since a straightforward computation shows that
\[
\varepsilon^{-3}J_{\varepsilon,A,V}(u)=J_{1,A_{\varepsilon},V_{\varepsilon}%
}(u_{\varepsilon})\text{ \ \ and \ \ }\varepsilon^{-3/2}\nabla_{\mathcal{N}%
_{\varepsilon,A,V}^{\tau}}J_{\varepsilon,A,V}(u)=\nabla_{\mathcal{N}%
_{1,A_{\varepsilon},V_{\varepsilon}}^{\tau}}J_{1,A_{\varepsilon}%
,V_{\varepsilon}}(u_{\varepsilon}),
\]
where $A_{\varepsilon}(x):=A(\varepsilon x)$ and $V_{\varepsilon
}(x):=V(\varepsilon x).$
\end{proof}

$\mathbb{S}^{1}$ acts on $H_{\varepsilon,A}^{1}(\mathbb{R}^{3},\mathbb{C}%
)^{\tau}$ by scalar multiplication: $(e^{\mathrm{i}\theta},u)\mapsto
e^{\mathrm{i}\theta}u.$ The Nehari manifold $\mathcal{N}_{\varepsilon
,A,V}^{\tau}$ and the functional $J_{\varepsilon,A,V}$ are invariant under
this action. Two solutions of (\ref{prob}) are geometrically distinct iff they
lie on different $\mathbb{S}^{1}$-orbits. Equivariant Lusternik-Schnirelmann
theory yields the following result, see e.g. \cite{cp}.

\begin{proposition}
\label{ls}If $J_{\varepsilon,A,V}:\mathcal{N}_{\varepsilon,A,V}^{\tau
}\rightarrow\mathbb{R}$ satisfies $(PS)_{c}$ at each level $c\leq\overline
{c},$ then $J_{\varepsilon,A,V}$ has at least%
\[
\text{\emph{cat}}\left[  \left(  \mathcal{N}_{\varepsilon,A,V}^{\tau}\cap
J_{\varepsilon,A,V}^{\overline{c}}\right)  /\mathbb{S}^{1}\right]
\]
critical $\mathbb{S}^{1}$-orbits in $\mathcal{N}_{\varepsilon,A,V}^{\tau}\cap
J_{\varepsilon,A,V}^{\overline{c}}$.
\end{proposition}

Here $(\mathcal{N}_{\varepsilon,A,V}^{\tau}\cap J_{\varepsilon,A,V}%
^{\overline{c}})/\mathbb{S}^{1}$ denotes the $\mathbb{S}^{1}$-orbit space of
$\mathcal{N}_{\varepsilon,A,V}^{\tau}\cap J_{\varepsilon,A,V}^{\overline{c}}$,
where, as usual, $J_{\varepsilon,A,V}^{c}:=\left\{  u\in H_{\varepsilon,A}%
^{1}(\mathbb{R}^{3},\mathbb{C}):J_{\varepsilon,A,V}(u)\leq c\right\}
$.\medskip

To prove Theorem \ref{mainthm1} we will show that%
\begin{equation}
\text{cat}_{B_{\rho}M_{\tau}/G}M_{\tau}/G\leq\text{ cat}\left[  \left(
\mathcal{N}_{\varepsilon,A,V}^{\tau}\cap J_{\varepsilon,A,V}^{d}\right)
/\mathbb{S}^{1}\right]  \label{ineqcat}%
\end{equation}
for some $d=d(\varepsilon)\in(c_{\varepsilon,A,V}^{\tau},\varepsilon^{3}%
\min_{x\in\mathbb{R}^{3}\setminus\{0\}}(\#Gx)V_{\infty}^{3/2}E_{1}),$ where%
\begin{equation}
c_{\varepsilon,A,V}^{\tau}:=\inf_{\mathcal{N}_{\varepsilon,A,V}^{\tau}%
}J_{\varepsilon,A,V}. \label{tauinf}%
\end{equation}
To obtain inequality (\ref{ineqcat}) we shall construct maps
\[
M_{\tau}/G\overset{\iota_{\varepsilon}}{\longrightarrow}\mathcal{C}%
/\mathbb{S}^{1}\overset{\beta_{\varepsilon}}{\longrightarrow}B_{\rho}M_{\tau
}/G,
\]
whose composition is the inclusion $M_{\tau}/G\hookrightarrow B_{\rho}M_{\tau
}/G,$ where $\mathcal{C}$ is a union of connected components of $\mathcal{N}%
_{\varepsilon,A,V}^{\tau}\cap J_{\varepsilon,A,V}^{d}.$ A standard argument
then yields%
\[
\text{cat}_{B_{\rho}M_{\tau}/G}M_{\tau}/G\leq\text{cat}\left(  \mathcal{C}%
/\mathbb{S}^{1}\right)  \leq\text{cat}\left[  \left(  \mathcal{N}%
_{\varepsilon,A,V}^{\tau}\cap J_{\varepsilon,A,V}^{d}\right)  /\mathbb{S}%
^{1}\right]  .
\]
The main ingredients for defining these maps are contained in the following
two sections.

\section{The entrance map}

\label{embedding}For any positive real number $\lambda$ we consider the
problem
\begin{equation}%
\begin{cases}
-\Delta u+\lambda u=(\frac{1}{\left\vert x\right\vert }\ast u^{2})u,\\
u\in H^{1}(\mathbb{R}^{3},\mathbb{R}).
\end{cases}
\label{limlambda}%
\end{equation}
Its associated energy functional $J_{\lambda}\colon H^{1}(\mathbb{R}%
^{3},\mathbb{R})\rightarrow\mathbb{R}$ is given by%
\[
J_{\lambda}(u)=\frac{1}{2}\left\Vert u\right\Vert _{\lambda}^{2}-\frac{1}%
{4}\mathbb{D}(u),\text{ \ \ \ with \ }\left\Vert u\right\Vert _{\lambda}%
^{2}:=\int_{\mathbb{R}^{3}}\left(  \left\vert \nabla u\right\vert ^{2}+\lambda
u^{2}\right)  .
\]
Its Nehari manifold will be denoted by
\[
\mathcal{M}{_{\lambda}:}=\left\{  u\in H^{1}(\mathbb{R}^{3},\mathbb{R}%
):\,u\neq0,\quad\left\Vert u\right\Vert _{\lambda}^{2}=\mathbb{D}(u)\right\}
.
\]
We set%
\[
E_{\lambda}:=\inf_{u\in\mathcal{M}{_{\lambda}}}J_{\lambda}(u).
\]
The critical points of $J_{\lambda}$ on $\mathcal{M}{_{\lambda}}$ are the
nontrivial solutions to (\ref{limlambda}). Note that $u$ solves (\ref{limit})
if and only if $u_{\lambda}(x):=\lambda u(\sqrt{\lambda}x)$ solves
(\ref{limlambda}). Therefore,%
\[
E_{\lambda}=\lambda^{3/2}E_{1}.
\]
Minimizers of $J_{\lambda}$ on $\mathcal{M}{_{\lambda}}$ are called ground
states. Lieb established in \cite{lieb} the existence and uniqueness of ground
states up to sign and translations. Recently Ma and Zhao \cite{mz} showed that
every positive solution to problem (\ref{limlambda}) is radially symmetric,
and they concluded from this fact that the positive solution to this
problem\ is unique up to translations. We denote by $\omega_{\lambda}$ the
positive solution to problem (\ref{limlambda}) which is radially symmetric
with respect to the origin.

Fix a radial function $\varrho\in C^{\infty}(\mathbb{R}^{3},\mathbb{R})$ such
that $\varrho(x)=1$ if $\left\vert x\right\vert \leq\frac{1}{2}$ and
$\varrho(x)=0$ if $\left\vert x\right\vert \geq1.$ For $\varepsilon>0$ set
$\varrho_{\varepsilon}(x):=\varrho(\sqrt{\varepsilon}x)$, $\omega
_{\lambda,\varepsilon}:=\varrho_{\varepsilon}\omega_{\lambda}$ $\ $and
\begin{equation}
\upsilon_{\lambda,\varepsilon}=\frac{\left\Vert \omega_{\lambda,\varepsilon
}\right\Vert _{\lambda}}{\sqrt{\mathbb{D}(\omega_{\lambda,\varepsilon})}%
}\,\omega_{\lambda,\varepsilon}. \label{bumps}%
\end{equation}
Note that $\operatorname{supp}(\upsilon_{\lambda,\varepsilon})\subset
B(0,1/\sqrt{\varepsilon}):=\{x\in\mathbb{R}^{3}:\left\vert x\right\vert
\leq1/\sqrt{\varepsilon}\}$ and $\upsilon_{\lambda,\varepsilon}\in
\mathcal{M}{_{\lambda}.}$ An easy computation shows that%
\begin{equation}
\lim_{\varepsilon\rightarrow0}J_{\lambda}(\upsilon_{\lambda,\varepsilon
})=\lambda^{3/2}E_{1}. \label{enerbumps}%
\end{equation}

Observe that
\[
\ell_{G,V}:=\inf_{x\in\mathbb{R}^{3}}(\#Gx)V^{3/2}(x)<V^{3/2}(0)<\infty.
\]
We assume from now on that there exists $\alpha>0$ such that the set%
\[
\left\{  y\in\mathbb{R}^{3}:(\#Gy)V^{3/2}(y)\leq\ell_{G,V}+\alpha\right\}
\]
is compact. Then%
\[
M_{G,V}:=\left\{  y\in\mathbb{R}^{3}:(\#Gy)V^{3/2}(y)=\ell_{G,V}\right\}
\]
is a compact $G$-invariant set and all $G$-orbits in $M_{G,V}$ are finite. We
split $M_{G,V}$ according to the orbit type of its elements as follows: we
choose subgroups $G_{1},\ldots,G_{m}$ of $G$ such that the isotropy subgroup
$G_{x}$ of every point $x\in M_{G,V}$ is conjugate to precisely one of the
$G_{i}$'s, and we set%
\[
M_{i}:=\left\{  y\in M_{G,V}:G_{y}=gG_{i}g^{-1}\text{ for some }g\in
G\right\}  .
\]
Since isotropy subgroups satisfy $G_{gx}=gG_{x}g^{-1},$ the sets $M_{i}$ are
$G$-invariant and, since $V$ is continuous, they are closed and pairwise
disjoint, and
\[
M_{G,V}=M_{1}\cup\cdots\cup M_{m}.
\]
Moreover, since%
\[
\left\vert G/G_{i}\right\vert V^{3/2}(y)=(\#Gy)V^{3/2}(y)=\ell_{G,V}\text{
\ \ \ for all \ }y\in M_{i},
\]
the potential $V$ is constant on each $M_{i}.$ Here $\left\vert G/G_{i}%
\right\vert $ denotes the index of $G_{i}$ in $G.$ We denote by $V_{i}$ the
value of $V$ on $M_{i}.$

Let $\upsilon_{i,\varepsilon}:=\upsilon_{V_{i},\varepsilon}$ be defined as in
(\ref{bumps}) with $\lambda:=V_{i}.$ For $\xi\in M_{i}$ set%
\[
\phi_{\varepsilon,\xi}(x):=\upsilon_{i,\varepsilon}\left(  \frac{x-\xi
}{\varepsilon}\right)  \exp\left(  -\mathrm{i}A(\xi)\cdot\left(  \frac{x-\xi
}{\varepsilon}\right)  \right)  .
\]
The proofs of the following two lemmas are similar to those of Lemmas 1 and 2
in \cite{cc1}, so we shall omit them.

\begin{lemma}
Uniformly in $\xi\in M_{i},$ we have that%
\[
\lim_{\varepsilon\rightarrow0}\varepsilon^{-3}J_{\varepsilon,A,V}\left[
\pi_{\varepsilon,A,V}(\phi_{\varepsilon,\xi})\right]  =V_{i}^{3/2}E_{1},
\]
where $\pi_{\varepsilon,A,V}$\ is as in \emph{(\ref{radproj})}.
\end{lemma}

It is well known that the map $G/G_{\xi}\rightarrow G\xi$ given by $gG_{\xi
}\mapsto g\xi$ is a homeomorphism, see e.g. \cite{tD}. So, if $G_{i}%
\subset\ker\tau$ and $\xi\in M_{i},$ then the map
\[
G\xi\rightarrow\mathbb{S}^{1},\ \ \ \ g\xi\mapsto\tau(g),
\]
is well defined and continuous. Set
\begin{equation}
\psi_{\varepsilon,\xi}(x):=\sum_{g\xi\in G\xi}\tau(g)\upsilon_{i,\varepsilon
}\left(  \frac{x-g\xi}{\varepsilon}\right)  e^{-\mathrm{i}A(g\xi)\cdot\left(
\frac{x-g\xi}{\varepsilon}\right)  }. \label{psi}%
\end{equation}

\begin{lemma}
\label{lemin}Assume that $G_{i}\subset\ker\tau$. Then, the following
hold:\newline(a) For every $\xi\in M_{i}$ and $\varepsilon>0,$ one has that%
\[
\psi_{\varepsilon,\xi}(gx)=\tau(g)\psi_{\varepsilon,\xi}(x)\text{
\ \ \ \ }\forall g\in G,\text{ }x\in\mathbb{R}^{3}.
\]
(b) For every $\xi\in M_{i}$ and $\varepsilon>0,$ one has that%
\[
\tau(g)\psi_{\varepsilon,g\xi}(x)=\psi_{\varepsilon,\xi}(x)\text{
\ \ \ \ }\forall g\in G,\text{ }x\in\mathbb{R}^{3}.
\]
(c) One has that%
\[
\lim_{\varepsilon\rightarrow0}{\varepsilon}^{-3}J_{\varepsilon,A,V}\left[
\pi_{\varepsilon,A,V}(\psi_{\varepsilon,\xi})\right]  =\ell_{G,V}E_{1}.
\]
uniformly in $\xi\in M_{i}.$
\end{lemma}

Let
\[
M_{\tau}:=\left\{  y\in M_{G,V}:G_{y}\subset\ker\tau\right\}  =\bigcup
_{G_{i}\subset\ker\tau}M_{i}.
\]

\begin{proposition}
\label{ingoing}The map $\widehat{\iota}_{\varepsilon}:M_{\tau}\rightarrow
\mathcal{N}_{\varepsilon,A,V}^{\tau}$ given by%
\[
\widehat{\iota}_{\varepsilon}(\xi):=\pi_{\varepsilon,A,V}(\psi_{\varepsilon
,\xi})
\]
is well defined and continuous, and satisfies%
\[
\tau(g)\widehat{\iota}_{\varepsilon}(g\xi)=\widehat{\iota}_{\varepsilon}%
(\xi)\text{ \ \ \ \ }\forall\xi\in M_{\tau},\text{ }g\in G.
\]
Moreover, given $d>\ell_{G}E_{1},$ there exists $\varepsilon_{d}>0$ such that%
\[
\varepsilon^{-3}\ J_{\varepsilon,A,V}(\widehat{\iota}_{\varepsilon}(\xi))\leq
d\text{ \ \ \ \ }\forall\xi\in M_{\tau},\text{ }\varepsilon\in(0,\varepsilon
_{d}).
\]

\end{proposition}

\begin{proof}
This follows immediately from Lemma \ref{lemin}.
\end{proof}

\section{A local baryorbit map}

\label{baryorbit}Let $W:\mathbb{R}^{3}\rightarrow\mathbb{R}$ be a bounded,
uniformly continuous function with $\inf_{\mathbb{R}^{3}}W>0$ and such that
$W(gx)=W(x)$ for all $g\in G$, $x\in\mathbb{R}^{3}$$.$ We assume that the set%
\begin{equation}
\left\{  y\in\mathbb{R}^{3}:(\#Gy)W^{3/2}(y)\leq\ell_{G,W}+\alpha\right\}
\label{compW}%
\end{equation}
is compact, where $\ell_{G,W}:=\inf_{x\in\mathbb{R}^{3}}(\#Gx)W^{3/2}(x),$ and
consider the real-valued problem%
\begin{equation}
\left\{
\begin{array}
[c]{l}%
-\varepsilon^{2}\Delta v+W(x)v=\frac{1}{\varepsilon^{2}}\left(  \frac
{1}{\left\vert x\right\vert }\ast u^{2}\right)  u,\\
v\in H^{1}(\mathbb{R}^{3},\mathbb{R}),\\
v(gx)=v(x)\text{ \ }\forall x\in\mathbb{R}^{3},\text{ }g\in G.
\end{array}
\right.  \label{real}%
\end{equation}
We write%
\[
\left\langle v,w\right\rangle _{\varepsilon,W}:=\int_{\mathbb{R}^{3}}\left(
\varepsilon^{2}\nabla v\cdot\nabla w+W(x)vw\right)  ,\quad\left\Vert
v\right\Vert _{\varepsilon,W}^{2}:=\int_{\mathbb{R}^{3}}\left(  \left\vert
\varepsilon\nabla v\right\vert ^{2}+W(x)v^{2}\right)  ,
\]
and set%
\[
H^{1}(\mathbb{R}^{3},\mathbb{R})^{G}:=\{v\in H^{1}(\mathbb{R}^{3}%
,\mathbb{R}):v(gx)=v(x)\text{ }\forall x\in\mathbb{R}^{3},\text{ }g\in G\}.
\]
The nontrivial solutions of (\ref{real}) are the critical points of the energy
functional
\[
J_{\varepsilon,W}(v)=\frac{1}{2}\left\Vert v\right\Vert _{\varepsilon,W}%
^{2}-\frac{1}{4\varepsilon^{2}}\mathbb{D}(v)
\]
on the Nehari manifold%
\[
\mathcal{M}_{\varepsilon,W}^{G}:=\{v\in H^{1}(\mathbb{R}^{3},\mathbb{R}%
)^{G}:v\neq0,\text{ }\left\Vert v\right\Vert _{\varepsilon,W}^{2}%
={\varepsilon^{-2}}\mathbb{D}(v)\}.
\]
Set%
\begin{equation}
c_{\varepsilon,W}^{G}:=\inf_{\mathcal{M}_{\varepsilon,W}^{G}}J_{\varepsilon
,W}=\inf_{\substack{v\in H^{1}(\mathbb{R}^{3},\mathbb{R})^{G}\\v\neq0}%
}\frac{\varepsilon^{2}\left\Vert v\right\Vert _{\varepsilon,W}^{4}%
}{4\mathbb{D}(v)}. \label{infepsV}%
\end{equation}

We wish to study the behavior of "minimizing sequences" for the family of
problems (\ref{real}), parametrized by $\varepsilon,$ as $\varepsilon
\rightarrow0.$ This is described in Proposition \ref{epsilonPS} below. We
start with some lemmas.

\begin{lemma}
\label{cotas}$0<(\inf_{\mathbb{R}^{3}}W)^{3/2}E_{1}\leq\varepsilon
^{-3}c_{\varepsilon,W}^{G}$ \ for every $\varepsilon>0,$ and%
\[
\limsup_{\varepsilon\rightarrow0}\varepsilon^{-3}c_{\varepsilon,W}^{G}\leq
\ell_{G,W}E_{1},
\]

\end{lemma}

\begin{proof}
Set $W_{0}:=\inf_{\mathbb{R}^{3}}W$ and write $v_{\varepsilon}%
(x):=v(\varepsilon x).$ Then $\left\Vert v_{\varepsilon}\right\Vert _{W_{0}%
}^{2}=\varepsilon^{-3}\left\Vert v\right\Vert _{\varepsilon,W_{0}}^{2}$ and
$\mathbb{D}(v_{\varepsilon})=\varepsilon^{-5}\mathbb{D}(v)$. If follows
immediately from (\ref{infepsV}) that
\[
W_{0}^{3/2}E_{1}\leq c_{1,W_{0}}^{G}=\varepsilon^{-3}c_{\varepsilon,W_{0}}%
^{G}\leq\varepsilon^{-3}c_{\varepsilon,W}^{G}.
\]
To prove the second inequality, take $\xi\in\mathbb{R}^{3}$ such that
$(\#G\xi)W^{3/2}(\xi)=\ell_{G,W}E_{1}$. Write $G\xi:=\{\xi_{1},...,\xi_{m}\}.$
Fix $0<\rho<\frac{1}{2}\min_{i\neq j}\left\vert \xi_{i}-\xi_{j}\right\vert ,$
and let $W_{\rho}:=\sup_{B(\xi_{1},\rho)}W$. Let $\upsilon_{\rho,\varepsilon
}:=\upsilon_{W_{\rho},\varepsilon}$ be defined as in (\ref{bumps}) with
$\lambda:=W_{\rho}.$ Set
\[
w_{\rho,\varepsilon}(x):=%
{\textstyle\sum\limits_{i=1}^{m}}
\upsilon_{\rho,\varepsilon}\left(  \frac{x-\xi_{i}}{\varepsilon}\right)  .
\]
If $\sqrt{\varepsilon}\leq\rho,$ then\ supp$(w_{\rho,\varepsilon})\subset
\cup_{i=1}^{m}B(\xi_{i},\rho)$. Therefore $w_{\rho,\varepsilon}\in
\mathcal{M}_{\varepsilon,W_{\rho}}^{G}$ and%
\[
\varepsilon^{-3}c_{\varepsilon,W}^{G}\leq\varepsilon^{-3}J_{\varepsilon
,W}(w_{\rho,\varepsilon})\leq\varepsilon^{-3}J_{\varepsilon,W\rho}%
(w_{\rho,\varepsilon})=mJ_{W_{\rho}}(\upsilon_{\rho,\varepsilon}).
\]
It follows from (\ref{enerbumps}) that%
\[
\limsup_{\varepsilon\rightarrow0}\varepsilon^{-3}c_{\varepsilon,W}^{G}\leq
mW_{\rho}^{3/2}E_{1}.
\]
Letting $\rho\rightarrow0,$ we conclude that%
\[
\limsup_{\varepsilon\rightarrow0}\varepsilon^{-3}c_{\varepsilon,W}^{G}%
\leq(\#G\xi)W^{3/2}(\xi)E_{1}=\ell_{G,W}E_{1},
\]
as claimed.
\end{proof}

\begin{lemma}
\label{wcont}Let $\varepsilon_{n}>0$ and $\xi_{n}\in\mathbb{R}^{3}$ such that
$\varepsilon_{n}\rightarrow0$ and $(W(\xi_{n}))$ converges. Set $\widehat
{W}_{n}(x):=W(\varepsilon_{n}x+\xi_{n})$ and $\widehat{W}:=\lim_{n\rightarrow
\infty}W(\xi_{n}).$ Then, for every sequence $(u_{n})$ in $H^{1}%
(\mathbb{R}^{3},\mathbb{R})$ such that $u_{n}\rightharpoonup u$ weakly in
$H^{1}(\mathbb{R}^{3},\mathbb{R})$ and every $w\in H^{1}(\mathbb{R}%
^{3},\mathbb{R}),$ the following hold:%
\[
\lim_{n\rightarrow\infty}\left(  \left\langle u_{n},w\right\rangle
_{1,\widehat{W}_{n}}-\left\langle u_{n}-u,w\right\rangle _{1,\widehat{W}_{n}%
}\right)  =\left\langle u,w\right\rangle _{1,\widehat{W}}%
\]
and%
\[
\lim_{n\rightarrow\infty}\left(  \left\Vert u_{n}\right\Vert _{1,\widehat
{W}_{n}}^{2}-\left\Vert u_{n}-u\right\Vert _{1,\widehat{W}_{n}}^{2}\right)
=\left\Vert u\right\Vert _{1,\widehat{W}}^{2}.
\]

\end{lemma}

\begin{proof}
The argument is similar for both equalities. We prove the second one. Since
$(u_{n})$ is bounded in $L^{2}(\mathbb{R}^{3})$ there exists $C>2\left\Vert
u\right\Vert _{L^{2}(\mathbb{R}^{3})}$ such that%
\begin{align*}
&  \left\vert \left\Vert u_{n}\right\Vert _{1,\widehat{W}_{n}}^{2}-\left\Vert
u_{n}-u\right\Vert _{1,\widehat{W}_{n}}^{2}-\left\Vert u\right\Vert
_{1,\widehat{W}}^{2}\right\vert \\
&  \leq\left\vert \left\Vert u_{n}\right\Vert _{1,\widehat{W}}^{2}-\left\Vert
u_{n}-u\right\Vert _{1,\widehat{W}}^{2}-\left\Vert u\right\Vert _{1,\widehat
{W}}^{2}\right\vert +\int_{\mathbb{R}^{3}}\left\vert (\widehat{W}_{n}%
-\widehat{W})(2u_{n}u-u^{2})\right\vert \\
&  \leq o(1)+C\left\Vert (\widehat{W}_{n}-\widehat{W})u\right\Vert
_{L^{2}(\mathbb{R}^{3})}.
\end{align*}
Given $\varepsilon>0$ we fix $R>0$ such that%
\[
\int_{\left\vert x\right\vert \geq R}(\widehat{W}_{n}-\widehat{W})^{2}%
u^{2}\leq(2\sup_{x\in\mathbb{R}^{3}}W)^{2}\int_{\left\vert x\right\vert \geq
R}u^{2}<\varepsilon^{2}.
\]
Since $W$ is uniformly continuous, there exists $\delta>0$ such that%
\[
\left\vert W(\varepsilon_{n}x+\xi_{n})-W(\xi_{n})\right\vert <\frac
{\varepsilon}{C}\text{ \ \ \ if }\left\vert x\right\vert <\frac{\delta
}{\varepsilon_{n}}.
\]
Fix $n_{0}\in\mathbb{N}$ such that $|W(\xi_{n})-\widehat{W}|<\frac
{\varepsilon}{C}$ \ and $\ \frac{\delta}{\varepsilon_{n}}>R$ \ if $n\geq
n_{0}.$ Then,%
\[
\int_{\left\vert x\right\vert \leq R}(\widehat{W}_{n}-\widehat{W})^{2}%
u^{2}<\varepsilon^{2}\text{ \ \ for all }n\geq n_{0}.
\]
Therefore,
\[
\lim_{n\rightarrow\infty}\left\Vert (\widehat{W}_{n}-\widehat{W})u\right\Vert
_{L^{2}(\mathbb{R}^{3})}=0.
\]
This concludes the proof.
\end{proof}

\begin{lemma}
\label{lemG}Let $(z_{n})$ be a sequence in $\mathbb{R}^{N}.$ Then, after
passing to a subsequence, there exist a closed subgroup $\Gamma$ of $G$ and a
sequence $(\zeta_{n})$ in $\mathbb{R}^{N}$such that\newline(a) \ $\left(
\text{\emph{dist}}(Gz_{n},\zeta_{n})\right)  $ is bounded,\newline(b)
\ $G_{\zeta_{n}}=\Gamma,$\newline(c) \ if $\left\vert G/\Gamma\right\vert
<\infty$ then $\left\vert g\zeta_{n}-\tilde{g}\zeta_{n}\right\vert
\rightarrow\infty$ for all $g,\tilde{g}\in G$ with $\tilde{g}g^{-1}%
\notin\Gamma$,\newline(d) \ if $\left\vert G/\Gamma\right\vert =\infty,$ there
exists a closed subgroup $\Gamma^{\prime}$ of $G$ such that $\Gamma
\subset\Gamma^{\prime},$ $\left\vert G/\Gamma^{\prime}\right\vert =\infty$ and
$\left\vert g\zeta_{n}-\tilde{g}\zeta_{n}\right\vert \rightarrow\infty$ for
all $g,\tilde{g}\in G$ with $\tilde{g}g^{-1}\notin\Gamma^{\prime}$.
\end{lemma}

\begin{proof}
See Lemma 3.2 in \cite{ccs2}.
\end{proof}

Set
\[
M_{G,W}:=\left\{  y\in\mathbb{R}^{3}:(\#Gy)W^{3/2}(y)=\ell_{G,W}\right\}
\]
Abusing notation we write again $G_{i}$ and $M_{i}$ for the groups and the
sets defined as in Section \ref{embedding} but now for $W$ instead of $V$. So
the value of $W$ on $M_{i}$ is constant and we denote it by $W_{i}.$ We fix
$\widehat{\rho}>0$ such that%
\begin{equation}%
\begin{array}
[c]{ll}%
\left\vert y-gy\right\vert >2\widehat{\rho} & \text{if \ }gy\neq y\in
M_{G,W},\\
\text{dist}(M_{i},M_{j})>2\widehat{\rho} & \text{if \ }i\neq j,
\end{array}
\label{robar}%
\end{equation}
For $\rho\in(0,\widehat{\rho}),$ let%
\[
M_{i}^{\rho}:=\{y\in\mathbb{R}^{3}:\text{dist}(y,M_{i})\leq\rho,\text{
\ }G_{y}=gG_{i}g^{-1}\text{ for some }g\in G\},
\]
and for each $\xi\in M_{i}^{\rho}$ and $\varepsilon>0,$ define%
\[
\theta_{\varepsilon,\xi}(x):=%
{\textstyle\sum\limits_{g\xi\in G\xi}}
\omega_{i}\left(  \frac{x-g\xi}{\varepsilon}\right)  ,
\]
where $\omega_{i}$ is unique positive ground state of problem (\ref{limlambda}%
) with $\lambda:=W_{i}$ which is radially symmetric with respect to the
origin. Set%
\[
\Theta_{\rho,\varepsilon}:=\{\theta_{\varepsilon,\xi}:\xi\in M_{1}^{\rho}%
\cup\cdot\cdot\cdot\cup M_{m}^{\rho}\}.
\]
The following holds.

\begin{proposition}
\label{epsilonPS}Let $\varepsilon_{n}>0$ and $v_{n}\in H^{1}(\mathbb{R}%
^{3},\mathbb{R})^{G}$ be such that%
\begin{equation}
\varepsilon_{n}\rightarrow0,\text{ \ \ }\varepsilon_{n}^{-3}J_{\varepsilon
_{n},W}(v_{n})\rightarrow\widehat{c},\text{ \ \ }\varepsilon_{n}%
^{-3}\left\Vert \nabla_{\varepsilon_{n}}J_{\varepsilon_{n},W}(v_{n}%
)\right\Vert _{\varepsilon_{n},W}^{2}\rightarrow0, \label{cs}%
\end{equation}
where $\widehat{c}:=\lim\inf_{\varepsilon\rightarrow0}\varepsilon
^{-3}c_{\varepsilon,W}^{G}$ \ and $\nabla_{\varepsilon_{n}}J_{\varepsilon
_{n},W}$ is the gradient of $J_{\varepsilon_{n},W}$ with respect to the scalar
product $\left\langle \cdot,\cdot\right\rangle _{\varepsilon_{n},W}.$ Then,
passing to a subsequence, there exist an $i\in\{1,...,m\}$ and a sequence
$(\xi_{n})$ in $\mathbb{R}^{3}$ such that\newline(i) \ $\ G_{\xi_{n}}=G_{i}%
,$\newline(ii) $\ \xi_{n}\rightarrow\xi\in M_{i},$\newline(iii) $\ \varepsilon
_{n}^{-3}\left\Vert \left\vert v_{n}\right\vert -\theta_{\varepsilon_{n}%
,\xi_{n}}\right\Vert _{\varepsilon_{n},W}^{2}\rightarrow0,$ \newline(iv)
\ $\widehat{c}=\lim_{\varepsilon\rightarrow0}\varepsilon^{-3}c_{\varepsilon
,W}^{G}=\ell_{G,W}E_{1}.$
\end{proposition}

\begin{proof}
A standard argument shows that the sequence $(\varepsilon_{n}^{-3}\Vert
v_{n}\Vert_{\varepsilon_{n},W}^{2})$ is bounded and that%
\[
\lim_{n\rightarrow\infty}\varepsilon_{n}^{-3}\Vert v_{n}\Vert_{\varepsilon
_{n},W}^{2}=\lim_{n\rightarrow\infty}\varepsilon_{n}^{-5}\mathbb{D}(v_{n}%
)={4}\widehat{c}=:c>0.
\]
Let $\widetilde{v}_{n}\in H^{1}(\mathbb{R}^{3},\mathbb{R})^{G}$ be given by
$\widetilde{v}_{n}(z):=v_{n}(\varepsilon_{n}z).$ Then,
\[
\Vert\widetilde{v}_{n}\Vert_{1,W_{_{n}}}^{2}=\varepsilon^{-3}\Vert{v}_{n}%
\Vert_{\varepsilon_{n},W}^{2}\text{ \ \ \ and \ \ \ }\mathbb{D}(\widetilde
{v}_{n})=\varepsilon_{n}^{-5}\mathbb{D}(v_{n}),
\]
where $W_{n}(z):=W(\varepsilon_{n}z).$ Set%
\[
\delta:=\limsup_{n\rightarrow\infty}\sup_{y\in\mathbb{R}^{3}}\int
_{B(y,1)}\left\vert \widetilde{v}_{n}\right\vert ^{2}.
\]
Since $c>0,$ Lions' lemma \cite[Lemma 1.21]{w}, together with inequality
(\ref{D}), yields that $\delta>0$. Choose $z_{n}\in\mathbb{R}^{3}$ such that
\[
\int_{B(z_{n},1)}\left\vert \widetilde{v}_{n}\right\vert ^{2}\geq\frac{\delta
}{2}%
\]
and replace $(z_{n})$ by a sequence $(\zeta_{n})$ having the properties stated
in Lemma \ref{lemG}. Set $\widehat{v}_{n}(z):=\widetilde{v}_{n}(z+\zeta_{n}).$
After passing to a subsequence, we may assume that $\widehat{v}_{n}%
\rightharpoonup\widehat{v}$ weakly in $H^{1}(\mathbb{R}^{3},\mathbb{R)},$
$\widehat{v}_{n}(x)\rightarrow\widehat{v}(x)$ a.e. on $\mathbb{R}^{3}$ and
$\widehat{v}_{n}\rightarrow\widehat{v}$ in $L_{loc}^{2}(\mathbb{R}%
^{3},\mathbb{R)}.$ Choosing $C\geq$ dist$\left(  \zeta_{n},Gz_{n}\right)  $
for all $n,$ we obtain%
\[
\int_{B(0,C+1)}\left\vert \widehat{v}_{n}\right\vert ^{2}=\int_{B(\zeta
_{n},C+1)}\left\vert \widetilde{v}_{n}\right\vert ^{2}\geq\int_{B(z_{n}%
,1)}\left\vert \widetilde{v}_{n}\right\vert ^{2}\geq\frac{\delta}{2}.
\]
Therefore, $\widehat{v}\neq0.$ \newline Set $\xi_{n}:=\varepsilon_{n}\zeta
_{n}$ and $\widehat{W}_{n}(x):=W(\varepsilon_{n}x+\xi_{n})$. Since $W$ is
bounded, a subsequence of $W(\xi_{n})$ converges. We set $\widehat{W}%
:=\lim_{n\rightarrow\infty}W(\xi_{n}).$ The weak continuity of $\mathbb{D}%
^{\prime}$ \cite[Lemma 3.5]{a}, together with Lemma \ref{wcont} and assumption
(\ref{cs}) imply that $\widehat{v}$ is a solution to problem (\ref{limlambda})
with $\lambda:=\widehat{W}.$ \newline Since $v_{n}$ and $W$ are $G$-invariant
we have that $\widehat{v}_{n}(g^{-1}x)=v_{n}(\varepsilon_{n}x+g\xi_{n}),$
$\widehat{W}_{n}(g^{-1}x)=W(\varepsilon_{n}x+g\xi_{n}),$ and $\widehat
{W}:=\lim_{n\rightarrow\infty}W(g\xi_{n})$ for each $g\in G.$ Fix
$g_{1},...,g_{k}\in G$ such that $\left\vert g_{i}\zeta_{n}-g_{j}\zeta
_{n}\right\vert \rightarrow\infty$ if $i\neq j.$ Then,
\begin{equation}
\widehat{v}_{n}g_{j}^{-1}-%
{\textstyle\sum\limits_{i=j+1}^{k}}
\widehat{v}g_{i}^{-1}(\cdot-g_{i}\zeta_{n}+g_{j}\zeta_{n})\rightharpoonup
\widehat{v}g_{j}^{-1}\label{wc}%
\end{equation}
weakly in $H^{1}(\mathbb{R}^{3},\mathbb{R}).$ Applying Lemma \ref{wcont} we
obtain%
\begin{align*}
&  \left\Vert \widehat{v}_{n}g_{j}^{-1}-%
{\textstyle\sum\limits_{i=j+1}^{k}}
\widehat{v}g_{i}^{-1}\left(  \cdot-g_{i}\zeta_{n}+g_{j}\zeta_{n}\right)
\right\Vert _{1,\widehat{W}_{n}g_{j}^{-1}}^{2}\\
&  =\left\Vert \widehat{v}_{n}g_{j}^{-1}-\widehat{v}g_{j}^{-1}-%
{\textstyle\sum\limits_{i=j+1}^{k}}
\widehat{v}g_{i}^{-1}\left(  \cdot-g_{i}\zeta_{n}+g_{j}\zeta_{n}\right)
\right\Vert _{1,\widehat{W}_{n}g_{j}^{-1}}^{2}+\left\Vert \widehat{v}%
g_{j}^{-1}\right\Vert _{1,\widehat{W}}^{2}+o(1),
\end{align*}
and performing the change of variable $y=\varepsilon_{n}x+g_{j}\xi_{n}$ we
conclude that%
\begin{align*}
&  \varepsilon_{n}^{-3}\left\Vert v_{n}-%
{\textstyle\sum\limits_{i=j+1}^{k}}
\widehat{v}g_{i}^{-1}\left(  \frac{\cdot-g_{i}\xi_{n}}{\varepsilon_{n}%
}\right)  \right\Vert _{\varepsilon_{n},W}^{2}\\
&  =\varepsilon_{n}^{-3}\left\Vert v_{n}-%
{\textstyle\sum\limits_{i=j}^{k}}
\widehat{v}g_{i}^{-1}\left(  \frac{\cdot-g_{i}\xi_{n}}{\varepsilon_{n}%
}\right)  \right\Vert _{\varepsilon_{n},W}^{2}+\left\Vert \widehat
{v}\right\Vert _{1,\widehat{W}}^{2}+o(1).
\end{align*}
Iterating these equalities we conclude that%
\[
4\widehat{c}=\lim_{n\rightarrow\infty}\varepsilon_{n}^{-3}\left\Vert
v_{n}\right\Vert _{\varepsilon_{n},W}^{2}=\lim_{n\rightarrow\infty}%
\varepsilon_{n}^{-3}\left\Vert v_{n}-\sum_{i=1}^{k}\widehat{v}g_{i}%
^{-1}\left(  \frac{\cdot-g_{i}\xi_{n}}{\varepsilon_{n}}\right)  \right\Vert
_{\varepsilon_{n},W}^{2}+k\left\Vert \widehat{v}\right\Vert _{1,\widehat{W}%
}^{2}.
\]
This implies that $4\widehat{c}\geq k\left\Vert \widehat{v}\right\Vert
_{1,\widehat{W}}^{2}$ which, together with property \emph{(d)} in Lemma
\ref{lemG}, implies $\left\vert G/\Gamma\right\vert <\infty.$ Property
\emph{(c)} allows us to take $k:=\left\vert G/\Gamma\right\vert $. Then,
property \emph{(b)} and Lemma \ref{cotas} yield
\begin{align*}
\ell_{G,W}E_{1} &  \leq\lim_{n\rightarrow\infty}(\#G\xi_{n})W^{3/2}(\xi
_{n})E_{1}=\left\vert G/\Gamma\right\vert \widehat{W}^{3/2}E_{1}\\
&  \leq\left\vert G/\Gamma\right\vert \frac{1}{4}\left\Vert \widehat
{v}\right\Vert _{1,\widehat{W}}^{2}\leq\widehat{c}\leq\limsup_{\varepsilon
\rightarrow0}\varepsilon^{-3}c_{\varepsilon,W}^{G}\leq\ell_{G,W}E_{1}.
\end{align*}
This proves \emph{(iv)} and gives also
\begin{equation}
\lim_{n\rightarrow\infty}\varepsilon_{n}^{-3}\left\Vert v_{n}-%
{\textstyle\sum\limits_{i=1}^{k}}
\widehat{v}g_{i}^{-1}\left(  \frac{\cdot-g_{i}\xi_{n}}{\varepsilon_{n}%
}\right)  \right\Vert _{\varepsilon_{n},W}^{2}=0.\label{aprox}%
\end{equation}
Moreover, $(\#G\xi_{n})W^{3/2}(\xi_{n})\leq\ell_{G,W}+\alpha$ for $n$ large
enough. Thus, assumption (\ref{comp}) implies, after passing to a subsequence,
that $\xi_{n}\rightarrow\xi.$ Hence, $W(\xi)=\widehat{W}$ and%
\[
\ell_{G,W}E_{1}\leq(\#G\xi)W(\xi)E_{1}\leq\left\vert G/\Gamma\right\vert
\widehat{W}^{3/2}E_{1}\leq\left\vert G/\Gamma\right\vert \frac{1}{4}\left\Vert
\widehat{v}\right\Vert _{1,\widehat{W}}^{2}\leq\ell_{G,W}E_{1}.
\]
We conclude that $\xi\in M_{i}$ for some $i=1,...,m$, as claimed in
\emph{(ii)}. Then, $\widehat{W}=W_{i}$, $\Gamma=G_{\xi}=gG_{i}g^{-1}$ for some
$g\in G,$ and $\widehat{v}$ is a ground state of problem (\ref{limlambda})
with $\lambda=W_{i}$. \newline Since the ground state is unique up to sign and
translation we must have that $\widehat{v}(z)=\pm\omega_{i}(z-z_{0})$ for some
$z_{0}\in\mathbb{R}^{3}.$ Observe that $\widehat{v}$ is $\Gamma$-invariant.
So, if $\Gamma$ is nontrivial, then $z_{0}=0$ and, since $\omega_{i}$ is
radial, equation (\ref{aprox}) becomes \emph{(iii)}. If, on the other hand,
$\Gamma$ is the trivial group, we replace $\xi_{n}$ by $\xi_{n}^{\prime}%
:=\xi_{n}+\varepsilon_{n}z_{0}.$ Since $G\xi_{n}\cong G$ and $\varepsilon
_{n}\rightarrow0,$ $\xi_{n}^{\prime}$ has the same properties as $\xi_{n}$ for
$n$ large enough. Moreover, since $\omega_{i}$ is radially symmetric,%
\[
\widehat{v}\left(  \frac{g^{-1}z-\xi_{n}}{\varepsilon_{n}}\right)  =\pm
\omega_{i}\left(  \frac{g^{-1}z-\xi_{n}^{\prime}}{\varepsilon_{n}}\right)
=\pm\omega_{i}\left(  \frac{z-g\xi_{n}^{\prime}}{\varepsilon_{n}}\right)
\]
and, again, equation (\ref{aprox}) yields \emph{(iii)}. This completes the proof.
\end{proof}

\begin{proposition}
\label{unique}Given $\rho\in(0,\widehat{\rho})$ there exist $d_{\rho}%
>\ell_{G,W}E_{1}$ and $\varepsilon_{\rho}>0$ with the following property: For
every $\varepsilon\in(0,\varepsilon_{\rho})$ and every $v\in\mathcal{M}%
_{\varepsilon,W}^{G}$ with $J_{\varepsilon,W}(v)\leq\varepsilon^{3}d_{\rho}$
there exists precisely one $G$-orbit $G\xi_{\varepsilon,v}$ with
$\xi_{\varepsilon,v}\in M_{1}^{\rho}\cup\cdot\cdot\cdot\cup M_{m}^{\rho}$ such
that%
\[
\varepsilon^{-3}\left\Vert \left\vert v\right\vert -\theta_{\varepsilon
,\xi_{\varepsilon,v}}\right\Vert _{\varepsilon,W}^{2}=\min_{\theta\in
\Theta_{\rho,\varepsilon}}\left\Vert \left\vert v\right\vert -\theta
\right\Vert _{\varepsilon,W}^{2}.
\]

\end{proposition}

\begin{proof}
The proof is analogous to that of Proposition 5.3 in \cite{cc2}. We omit the details.
\end{proof}

Fix $\rho\in(0,\widehat{\rho})$ and $\varepsilon\in(0,\varepsilon_{\rho}).$
Proposition \ref{unique} allows us to define a map%
\begin{equation}
\widehat{\beta}_{\rho,\varepsilon,0}:\mathcal{M}_{\varepsilon,W}^{G}\cap
J_{\varepsilon,W}^{\varepsilon^{3}d_{\rho}}\longrightarrow\left(  M_{1}^{\rho
}\cup\cdot\cdot\cdot\cup M_{m}^{\rho}\right)  /G \label{bar}%
\end{equation}
by taking%
\[
\widehat{\beta}_{\rho,\varepsilon,0}(v):=G\xi_{\varepsilon,v}.
\]
Here, as usual, $J_{\varepsilon,W}^{c}:=\{v\in H^{1}(\mathbb{R}^{3}%
,\mathbb{R}):J_{\varepsilon,W}(v)\leq c\}.$ The map $\widehat{\beta}%
_{\rho,\varepsilon,0}$ is the $G$-equivariant analogon to the usual baricenter
map. It is only defined for functions in $\mathcal{M}_{\varepsilon,W}^{G}$
with small enough energy. We call it the \emph{local baryorbit map.} It
reflects the fact that such functions concentrate at a unique $G$-orbit with
minimal cardinality as $\varepsilon\rightarrow0.$

\section{Proofs of the main results}

\label{secproofs}Let $V_{\infty}:=\liminf_{\left\vert x\right\vert
\rightarrow\infty}V(x).$ Assumption (\ref{comp}) implies that
\[
\ell_{G,V}<\min_{x\in\mathbb{R}^{3}\setminus\{0\}}(\#Gx)V_{\infty}^{3/2}.
\]
We fix $\delta_{0}>0$ and $\lambda\in(0,V_{\infty})$ such that
\begin{equation}
\ell_{G,V}E_{1}+\delta_{0}<\min_{x\in\mathbb{R}^{3}\setminus\{0\}}%
(\#Gx)\lambda^{3/2}E_{1}<\min_{x\in\mathbb{R}^{3}\setminus\{0\}}%
(\#Gx)V_{\infty}^{3/2}E_{1}, \label{delta}%
\end{equation}
and define $W(x):=\min\{V(x),\lambda\}.$ This $W$ has all properties stated in
section \ref{baryorbit}, in particular, it is uniformly continuous. Moreover,
$\ell_{G,W}=\ell_{G,V}$ and $M_{G,W}=M_{G,V}.$

Let $\pi_{\varepsilon,W}:H^{1}(\mathbb{R}^{3},\mathbb{R})^{G}\setminus
\{0\}\rightarrow\mathcal{M}_{\varepsilon,W}^{G}$ \ denote the radial
projection onto the Nehari manifold, which is given by%
\begin{equation}
\pi_{\varepsilon,W}(v):=\frac{\varepsilon\left\Vert u\right\Vert
_{\varepsilon,W}}{\sqrt{\mathbb{D}(u)}}v. \label{rad}%
\end{equation}
Observe that%
\begin{equation}
J_{\varepsilon,W}(\pi_{\varepsilon,W}(v))=\frac{\varepsilon^{2}\left\Vert
v\right\Vert _{\varepsilon,W}^{4}}{4\mathbb{D}(v)}\text{ \ \ \ \ for all }v\in
H^{1}(\mathbb{R}^{3},\mathbb{R})^{G}\setminus\{0\}. \label{Jpi}%
\end{equation}
Let $\widehat{\iota}_{\varepsilon}$ be the map defined in Proposition
\ref{ingoing} and $\widehat{\beta}_{\rho,\varepsilon,0}$\ be as in
(\ref{bar}). Then, for $d_{\rho}>\ell_{G,V}E_{1}$ and $\varepsilon_{\rho}>0$
as in Proposition \ref{unique} the following holds.

\begin{proposition}
\label{bary}For each $\rho\in(0,\widehat{\rho})$ and $\varepsilon
\in(0,\varepsilon_{\rho}),$ the map
\[
\widehat{\beta}_{\rho,\varepsilon}\colon\mathcal{N}_{\varepsilon,A,V}^{\tau
}\cap J_{\varepsilon,A,V}^{\varepsilon^{3}d_{\rho}}\rightarrow\left(
M_{1}^{\rho}\cup\cdot\cdot\cdot\cup M_{m}^{\rho}\right)  /G,\hspace
{0.3in}\widehat{\beta}_{\rho,\varepsilon}(u):=\widehat{\beta}_{\rho
,\varepsilon,0}(\pi_{\varepsilon,W}(\left\vert u\right\vert )),
\]
is well defined and continuous and satisfies\newline(i) $\ \widehat{\beta
}_{\rho,\varepsilon}(\gamma u)=\widehat{\beta}_{\rho,\varepsilon}(u)$ for all
$\gamma\in\mathbb{S}^{1}$,\newline(ii) $\widehat{\beta}_{\rho,\varepsilon
}(\widehat{\iota}_{\varepsilon}(\xi))=G\xi$ for all $\xi\in M_{\tau}$ with
$J_{\varepsilon,A,V}(\widehat{\iota}_{\varepsilon}(\xi))\leq\varepsilon
^{3}d_{\rho}$.
\end{proposition}

\begin{proof}
If $u\in\mathcal{N}_{\varepsilon,A,V}^{\tau}$ then $\left\vert u\right\vert
\in H^{1}(\mathbb{R}^{3},\mathbb{R})^{G}\setminus\{0\}$ and, since $W\leq V,$
formulas (\ref{Jpi}) and (\ref{JpiA}), together with the diamagnetic
inequality (\ref{di}) yield%
\begin{equation}
J_{\varepsilon,W}(\pi_{\varepsilon,W}(\left\vert u\right\vert ))\leq
J_{\varepsilon,V}(\pi_{\varepsilon,V}(\left\vert u\right\vert ))\leq
J_{\varepsilon,A,V}(u). \label{inf}%
\end{equation}
So $J_{\varepsilon,W}(\pi_{\varepsilon,W}(\left\vert u\right\vert
))\leq\varepsilon^{3}d_{\rho}$ if $J_{\varepsilon,A,V}(u)\leq\varepsilon
^{3}d_{\rho}.$ Therefore, $\widehat{\beta}_{\rho,\varepsilon}$ is well
defined. It is straightforward to verify that it has the desired properties.
\end{proof}

Let%
\[
M_{\tau}^{\rho}:=%
{\textstyle\bigcup\limits_{G_{i}\subset\ker\tau}}
M_{i}^{\rho}.
\]
Propositions \ref{ingoing}\ and \ref{bary}\ allow us to estimate the
Lusternik-Schnirelmann category of low energy sublevel sets as follows.

\begin{corollary}
\label{cat}For every $\rho\in(0,\widehat{\rho})$ and $d\in(\ell_{G,V}%
E_{1},d_{\rho})$ there exists $\varepsilon_{\rho,d}>0$ such that%
\[
\text{\emph{cat}}_{M_{\tau}^{\rho}/G}M_{\tau}/G\leq\text{\emph{cat}}\left(
(\mathcal{N}_{\varepsilon,A,V}^{\tau}\cap J_{\varepsilon,A,V}^{\varepsilon
^{3}d})/\mathbb{S}^{1}\right)
\]
for every $\varepsilon\in(0,\varepsilon_{\rho,d}).$
\end{corollary}

\begin{proof}
Set $\varepsilon_{\rho,d}:=\min\{\varepsilon_{d},\varepsilon_{\rho}\}$ where
$\varepsilon_{d}$\ is as in Proposition \ref{ingoing}. Fix $\varepsilon
\in(0,\varepsilon_{\rho,d}).$ Then,%
\[
J_{\varepsilon,A,V}(\widehat{\iota}_{\varepsilon}(\xi))\leq\varepsilon
^{3}d\text{ \ \ and \ }\widehat{\beta}_{\rho,\varepsilon}(\widehat{\iota
}_{\varepsilon}(\xi))=\xi\text{ \ \ for all }\xi\in M_{\tau}.
\]
Since $M_{1}^{\rho},...,M_{m}^{\rho}$ are $G$-invariant and pairwise disjoint,
the set
\[
\mathcal{C}:=\{u\in\mathcal{N}_{\varepsilon,A,V}^{\tau}\cap J_{\varepsilon
,A,V}^{\varepsilon^{3}d}:\widehat{\beta}_{\rho,\varepsilon}(u)\in M_{\tau
}^{\rho}/G\}
\]
is a union of connected components of $\mathcal{N}_{\varepsilon,A,V}^{\tau
}\cap J_{\varepsilon,A,V}^{\varepsilon^{3}d}.$ Therefore,%
\[
\text{cat}\left(  \mathcal{C}/\mathbb{S}^{1}\right)  \leq\text{cat}\left(
(\mathcal{N}_{\varepsilon,A,V}^{\tau}\cap J_{\varepsilon,A,V}^{\varepsilon
^{3}d})/\mathbb{S}^{1}\right)  .
\]
By Propositions \ref{ingoing} and \ref{bary}, the maps%
\[
M_{\tau}/G\overset{\iota_{\varepsilon}}{\longrightarrow}\mathcal{C}%
/\mathbb{S}^{1}\overset{\beta_{\rho,\varepsilon}}{\longrightarrow}M_{\tau
}^{\rho}/G,
\]
given by $\iota_{\varepsilon}(G\xi):=\widehat{\iota}_{\varepsilon}(\xi)$ and
$\beta_{\rho,\varepsilon}(\mathbb{S}^{1}u):=\widehat{\beta}_{\rho,\varepsilon
}(u),$ are well defined and satisfy $\beta_{\rho,\varepsilon}(\iota
_{\varepsilon}(\xi))=\xi$ for all\ $\xi\in M_{\tau}.$ Therefore,%
\[
\text{cat}_{M_{\tau}^{\rho}/G}M_{\tau}/G\leq\text{cat}\left(  \mathcal{C}%
/\mathbb{S}^{1}\right)  .
\]
This finishes the proof.
\end{proof}

Another consequence of our previous results is the following.

\begin{corollary}
\label{infepsilon}If there exists $\xi\in\mathbb{R}^{3}$ such that
$(\#G\xi)V^{3/2}(\xi)=\ell_{G,V}$ and $G_{\xi}\subset\ker\tau$, then%
\[
\lim_{\varepsilon\rightarrow\infty}\varepsilon^{-3}c_{\varepsilon,A,V}^{\tau
}=\ell_{G,V}E_{1},
\]
where $c_{\varepsilon,A,V}^{\tau}:=\inf_{\mathcal{N}_{\varepsilon,A,V}^{\tau}%
}J_{\varepsilon,A,V}.$
\end{corollary}

\begin{proof}
Inequality (\ref{inf}) yields $c_{\varepsilon,W}^{G}:=\inf_{\mathcal{M}%
_{\varepsilon,W}^{G}}J_{\varepsilon,W}\leq\inf_{\mathcal{N}_{\varepsilon
,A,V}^{\tau}}J_{\varepsilon,A,V}=:c_{\varepsilon,A,V}^{\tau}.$ By Proposition
\ref{epsilonPS} and Lemma \ref{lemin}\emph{(c)},%
\[
\ell_{G,W}E_{1}=\lim_{\varepsilon\rightarrow\infty}\varepsilon^{-3}%
c_{\varepsilon,W}^{G}\leq\liminf_{\varepsilon\rightarrow0}\varepsilon
^{-3}c_{\varepsilon,A,V}^{\tau}\leq\limsup_{\varepsilon\rightarrow\infty
}\varepsilon^{-3}c_{\varepsilon,A,V}^{\tau}\leq\ell_{G,V}E_{1}.
\]
Since $\ell_{G,W}E_{1}=\ell_{G,V}E_{1},$ our claim follows.
\end{proof}

\bigskip

\noindent\textbf{Proof of Theorem \ref{mainthm1}.}\qquad Let $\rho,\delta>0$
be given. We may assume that $\rho\in(0,\widehat{\rho})$ with $\widehat{\rho}$
as in (\ref{robar}) and that $\delta\in(0,\delta_{0})$ with $\delta_{0}$ as in
(\ref{delta}). By Corollary \ref{infepsilon} there exists $\varepsilon
_{\delta}>0$ such that%
\[
\ell_{G,V}E_{1}-\delta<\varepsilon^{-3}c_{\varepsilon,A,V}^{\tau}%
\text{\hspace{0.3in}for all }\varepsilon\in(0,\varepsilon_{\delta}).
\]
Fix $d\in(\ell_{G,V}E_{1},\min\{d_{\rho},\ell_{G,V}E_{1}+\delta\})$ and set
$\widehat{\varepsilon}:=\min\{\varepsilon_{\delta},\varepsilon_{\rho,d}\}$
with $\varepsilon_{\rho,d}$ as in Corollary \ref{cat}. Since (\ref{delta})
holds, Proposition \ref{palaissmale} asserts that $J_{\varepsilon,A,V}%
\colon\mathcal{N}_{\varepsilon,A,V}^{\tau}\rightarrow\mathbb{R}$ satisfies
$(PS)_{c}$ for every $c\leq\varepsilon^{3}d$. Applying Proposition \ref{ls}
and Corollary \ref{cat} we conclude that $J_{\varepsilon,A,V}$ has at least%
\[
\text{cat}_{M_{\tau}^{\rho}/G}M_{\tau}/G
\]
geometrically distinct solutions $u\in\mathcal{N}_{\varepsilon,A,V}^{\tau}$
satisfying%
\[
\varepsilon^{3}\ell_{G,V}E_{1}-\varepsilon^{3}\delta<J_{\varepsilon
,A,V}(u)=\frac{1}{4\varepsilon^{2}}\mathbb{D}(u)\leq\varepsilon^{3}%
d<\varepsilon^{3}\ell_{G,V}E_{1}+\varepsilon^{3}\delta,
\]
for each $\varepsilon\in(0,\widehat{\varepsilon})$, as claimed. \qed\noindent

\bigskip

\noindent\textbf{Proof of Theorem \ref{concentrationthm}.}\qquad After passing
to a subsequence, we may assume that $\ell_{G,V}E_{1}-\frac{1}{2n}%
\leq\varepsilon_{n}^{-3}c_{\varepsilon_{n},W}^{G}$ and $\delta_{n}\leq\frac
{1}{2n}.$ Then, inequality (\ref{inf}) yields%
\begin{equation}
c_{\varepsilon_{n},W}^{G}\leq J_{\varepsilon_{n},W}(\pi_{\varepsilon_{n}%
,W}\left(  \left\vert u_{n}\right\vert \right)  )\leq J_{\varepsilon_{n}%
,A,V}(u_{n})\leq\varepsilon_{n}^{3}(\ell_{G}E_{1}+\delta_{n})\leq
c_{\varepsilon_{n},W}^{G}+\varepsilon_{n}^{3}/n. \label{inc}%
\end{equation}
By Ekeland's variational principle \cite[Theorem 8.5]{w} we may choose
$v_{n}\in\mathcal{M}_{\varepsilon_{n},W}^{G}$ such that%
\begin{equation}
\varepsilon_{n}^{-3}\left\Vert \pi_{\varepsilon_{n},W}\left(  \left\vert
u_{n}\right\vert \right)  -v_{n}\right\Vert _{\varepsilon_{n},W}%
^{2}\rightarrow0, \label{ic1}%
\end{equation}%
\[
\varepsilon_{n}^{-3}J_{\varepsilon_{n},W}(v_{n})\rightarrow\ell_{G}E_{1}%
\quad\text{and}\quad\varepsilon_{n}^{-3}\Vert\nabla_{\varepsilon_{n}%
}J_{\varepsilon_{n},W}(v_{n})\Vert_{\varepsilon_{n},W}^{2}\rightarrow0.
\]
After passing again to a subsequence, Proposition \ref{epsilonPS} gives a
sequence $(\xi_{n})$ in $\mathbb{R}^{3}$ such that $\xi_{n}\rightarrow\xi\in
M_{\tau}$, $G_{\xi_{n}}=G_{\xi},$ and%
\begin{equation}
\varepsilon_{n}^{-3}\left\Vert v_{n}-%
{\textstyle\sum\limits_{g\xi_{n}\in G\xi_{n}}}
\omega_{\xi}\left(  \frac{\cdot-g\xi_{n}}{\varepsilon_{n}}\right)  \right\Vert
_{\varepsilon_{n},W}^{2}\rightarrow0. \label{ic2}%
\end{equation}
Since $u_{n}\in\mathcal{N}_{\varepsilon_{n},A,V}^{\tau},$ multiplying
inequality (\ref{inc}) by $4\varepsilon_{n}^{2}\left[  \mathbb{D}%
(|u_{n}|)\right]  ^{-1}=4\varepsilon_{n}^{2}\left[  \mathbb{D}(u_{n})\right]
^{-1}=\left[  J_{\varepsilon_{n},A,V}(u_{n})\right]  ^{-1},$ using (\ref{Jpi})
and (\ref{JpiA}), and observing that $\varepsilon_{n}^{-3}J_{\varepsilon
_{n},A,V}(u_{n})\rightarrow\ell_{G}E_{1},$ we get%
\begin{align*}
\left\vert 1-\left(  \frac{\varepsilon_{n}\Vert\left\vert u_{n}\right\vert
\Vert_{\varepsilon_{n},W}}{\sqrt{\mathbb{D}(|u_{n}|)}}\right)  ^{4}%
\right\vert  &  =\left\vert \left(  \frac{\varepsilon_{n}^{2}\Vert u_{n}%
\Vert_{\varepsilon_{n},A,V}^{2}}{\mathbb{D}(u_{n})}\right)  ^{2}-\left(
\frac{\varepsilon_{n}^{2}\Vert\left\vert u_{n}\right\vert \Vert_{\varepsilon
_{n},W}^{2}}{\mathbb{D}(|u_{n}|)}\right)  ^{2}\right\vert \\
&  \leq\left(  \varepsilon_{n}^{-3}J_{\varepsilon_{n},A,V}(u_{n})\right)
^{-1}\frac{1}{n}\rightarrow0.
\end{align*}
Recalling (\ref{rad}) and using the diamagnetic inequality (\ref{di}), we then
obtain%
\begin{align}
\varepsilon_{n}^{-3}\Vert\left\vert u_{n}\right\vert -\pi_{\varepsilon_{n}%
,W}\left(  \left\vert u_{n}\right\vert \right)  \Vert_{\varepsilon_{n},W}^{2}
&  =\left\vert 1-\frac{\varepsilon_{n}\Vert\left\vert u_{n}\right\vert
\Vert_{\varepsilon_{n},W}}{\sqrt{\mathbb{D}(|u_{n}|)}}\right\vert
^{2}\varepsilon_{n}^{-3}\Vert\left\vert u_{n}\right\vert \Vert_{\varepsilon
_{n},W}^{2}\nonumber\\
&  \leq\left\vert 1-\frac{\varepsilon_{n}\Vert\left\vert u_{n}\right\vert
\Vert_{\varepsilon_{n},W}}{\sqrt{\mathbb{D}(|u_{n}|)}}\right\vert
^{2}4\varepsilon_{n}^{-3}J_{\varepsilon_{n},A,V}(u_{n})\rightarrow0.
\label{ic3}%
\end{align}
Finally, combining (\ref{ic1}), (\ref{ic2}), and (\ref{ic3}) we conclude that%
\[
\varepsilon_{n}^{-3}\left\Vert \left\vert u_{n}\right\vert -%
{\textstyle\sum\limits_{g\xi_{n}\in G\xi_{n}}}
\omega_{\xi}\left(  \frac{\cdot-g\xi_{n}}{\varepsilon_{n}}\right)  \right\Vert
_{\varepsilon_{n},W}^{2}\rightarrow0.
\]
Since $\left\Vert v\right\Vert _{\varepsilon}^{2}\leq C\left\Vert v\right\Vert
_{\varepsilon,W}^{2}$ for some constant independent of $\varepsilon,$ our
claim follows. \qed\noindent



\begin{thebibliography}{99}                                                                                               %


\bibitem {at}L. Abatangelo, S. Terracini, \emph{Solutions to nonlinear
Schr\"{o}dinger equations with singular electromagnetic potential and critical
exponent}, preprint 2010.

\bibitem {a}N. Ackermann, \emph{On a periodic Schr\"{o}dinger equation with
nonlocar superlinear part}, Math. Z. \textbf{248} (2004), 423-443.

\bibitem {cc1}S. Cingolani, M. Clapp, \emph{Intertwining semiclassical bound
states to a nonlinear magnetic equation}, Nonlinearity \textbf{22} (2009), 2309-2331.

\bibitem {cc2}S. Cingolani, M. Clapp, \emph{Symmetric semiclassical states to
a magnetic nonlinear Schr\"{o}dinger equation via Equivariant Morse Theory},
Commun. Pure Appl. Anal. \textbf{9} (2010), 1263--1281.

\bibitem {ccs2}S. Cingolani, M. Clapp, S. Secchi, \emph{Multiple solutions to
a magnetic nonlinear Choquard equation,} Z. Angew. Math. Phys. at press.

\bibitem {css}S. Cingolani, S. Secchi, M. Squassina, \emph{Semiclassical limit
for Schr\"{o}dinger equations with magnetic field and Hartree-type
nonlinearities}, Proc. Roy. Soc. Edinburgh, \textbf{140 A} (2010), 973--1009.

\bibitem {cp}M. Clapp, D. Puppe, \emph{Critical point theory with symmetries},
J. reine angew. Math. \textbf{418} (1991), 1-29.

\bibitem {cs}M. Clapp, A. Szulkin, \emph{Multiple solutions to a nonlinear
Schr\"{o}dinger equation with Aharonov-Bohm magnetic potential}, Nonlinear
Differ. Equ. Appl. (NoDea) \textbf{17} (2010), 229-248.

\bibitem {tD}T. tom Dieck, "Transformation groups", Walter de Gruyter,
Berlin-New York 1987.

\bibitem {fl}J. Fr\"{o}hlich, E. Lenzmann, \emph{Mean-field limit of quantum
Bose gases and nonlinear Hartree equation}, in S\'{e}minaire: \'{E}quations
aux D\'{e}riv\'{e}es Partielles 2003--2004, Exp. No. XIX, 26 pp., S\'{e}min.
\'{E}qu. D\'{e}riv. Partielles, \'{E}cole Polytech., Palaiseau, 2004.

\bibitem {fty}J. Fr\"{o}hlich, T.-P. Tsai, H.-T. Yau, \emph{On the
point-particle (Newtonian) limit of the non-linear Hartree equation}, Comm.
Math. Phys. \textbf{225} (2002), 223--274.

\bibitem {hmt}R. Harrison, I. Moroz, K.P. Tod, \emph{A numerical study of the
Schr\"{o}dinger-Newton equations}, Nonlinearity \textbf{16} (2003), 101--122.

\bibitem {lieb}E.H. Lieb, \emph{Existence and uniqueness of the minimizing
solution of Choquard's nonlinear equation}, Stud. Appl. Math. \textbf{57}
(1977), 93-105.

\bibitem {ll}E.H. Lieb, M. Loss, "Analysis", Graduate Studies in Math.
\textbf{14}, Amer. Math. Soc. 1997.

\bibitem {l}P.-L. Lions, \emph{The concentration-compacteness principle in the
calculus of variations. The locally compact case}, Ann. Inst. Henry
Poincar\'{e}, Analyse Non Lin\'{e}aire \textbf{1} (1984), 109-145 and 223-283.

\bibitem {l2}P.-L. Lions, \emph{The Choquard equation and related questions,}
Nonlinear Anal. T.M.A. \textbf{4} (1980), 1063--1073.

\bibitem {mz}L. Ma, L. Zhao, \emph{Classification of positive solitary
solutions of the nonlinear Choquard equation}, Arch. Rational Mech. Anal.
\textbf{195} (2010), 455-467.

\bibitem {mpt}I.M. Moroz, R. Penrose, P. Tod, \emph{Spherically-symmetric
solutions of the Schr\"{o}dinger-Newton equations}, Topology of the Universe
Conference (Cleveland, OH, 1997), Classical Quantum Gravity \textbf{15}
(1998), 2733--2742.

\bibitem {tm}I.M. Moroz, P. Tod, \emph{An analytical approach to the
Schr\"{o}dinger-Newton equations}, Nonlinearity \textbf{12} (1999), 201--216.

\bibitem {n}M. Nolasco, \emph{Breathing modes for the Schr\"{o}dinger--Poisson
system with a multiple--well external potential}, Commun. Pure Appl. Anal.
\textbf{9}, No.~5 (2010), 1411--1419.

\bibitem {p}R. Palais, \emph{The principle of symmetric criticallity}, Comm.
Math. Phys. \textbf{69} (1979), 19-30.

\bibitem {pe1}R. Penrose, \emph{On gravity's role in quantum state reduction},
Gen. Rel. Grav. \textbf{28} (1996), 581--600.

\bibitem {pe2}R. Penrose, \emph{Quantum computation, entanglement and state
reduction}, R. Soc. Lond. Philos. Trans. Ser. A Math. Phys. Eng. Sci.
\textbf{356} (1998), 1927--1939.

\bibitem {pe3}R. Penrose, "The road to reality. A complete guide to the laws
of the universe", Alfred A. Knopf Inc., New York (2005).

\bibitem {secchi}S. Secchi, \emph{A note on Schr\"{o}dinger--Newton systems
with decaying electric potential}, Nonlinear Analysis \textbf{72} (2010), 3842--3856.

\bibitem {t}P. Tod, \emph{The ground state energy of the
Schr\"{o}dinger-Newton equation}, Physics Letters A \textbf{280} (2001), 173--176.

\bibitem {wei}J. Wei, M. Winter, \emph{Strongly interacting bumps for the
Schr\"{o}dinger--Newton equation}, J. Math. Phys. \textbf{50} (2009), 012905.

\bibitem {w}M. Willem, "Minimax theorems", PNLDE \textbf{24}, Birkh\"{a}user,
Boston-Basel-Berlin 1996.
\end{thebibliography}
\end{document}